\documentclass{amsart}
\usepackage{amsmath, amssymb, amsthm, amscd, amsfonts, dsfont, mathrsfs}
\usepackage{color}
\newtheorem{theorem}{Theorem}

\theoremstyle{theorem}
\newtheorem{lemma}[theorem]{Lemma}

\newtheorem{prop}[theorem]{Proposition}

\newtheorem{corollary}[theorem]{Corollary}

\newtheorem*{Theorem*}{Theorem}
\newtheorem*{prop*}{Propostion}

\theoremstyle{definition}
\newtheorem{remark}[theorem]{Remark}
\newtheorem{question}{Question}
\newtheorem{definition}[theorem]{Definition}

\def\om{\omega} 
\def\Om{\Omega}
\def\C{{\mathbb C}}
\def\R{{\mathbb R}}
\newcommand{\1}{{\sf 1\hspace*{-0.8ex}\sf 1 }}

\title[Equilibrium measures]{Equilibrium measures on Julia sets of random quadratic polynomials}
\author{Krzysztof Lech}
 \address{Institute of Mathematics, University of Warsaw, ul. Banacha 2, 02-097 Warszawa, Poland}
 \email{K.Lech@mimuw.edu.pl}
\author{Anna Zdunik}
 \address{Institute of Mathematics, University of Warsaw, ul. Banacha 2, 02-097 Warszawa, Poland}
\email{A.Zdunik@mimuw.edu.pl}
\begin{document}
\begin{abstract}
    We consider sequences of compositions of quadratic polynomials $f_{c_n} (z) = z^2 + c_n$. For such sequences one can naturally generalize the definitions of the Julia set and basin of infinity from the autonomous case. In this setting the Julia set depends on a sequence $\omega = (c_0, c_1, ...)$.
We study the  equilibrium (harmonic) measure  on such Julia sets. In particular, we calculate the Hausdorff dimension of the equilibrium measure and study its dependence on the ''scale of randomness''. 
\end{abstract}

\maketitle

\section{Introduction. The results.}\label{sec:intro}

In recent decades there has been a rapid growth of interest in non-autonomous and random dynamical systems. The monography \cite{Arnold} contains an excellent study of the formalism of such systems and provides a number of fundamental results.
In the setting of random holomorphic dynamics, the foundations of the theore were laid in the seminal paper \cite{FS}.

A systematic study of random dynamics of quadratic polynomial was originated in the papers of Br\"uck, B\"uger and Reitz (see, e.g., \cite{B}, \cite{bruck}, \cite{BBR}.

More recently, Sumi obtained a number of very interesting results concerning action of semigroups of rational maps in both topological and probabilistic context (e.g., \cite{sumi}, \cite{sumi2}).

The recent paper \cite{LZ}  deals with random iteration of quadratic polynomials, 
in particular answering an old question whether the Julia set of a random iteration of quadratic polynomials is typically totally disconnected. 

The study of random iteration of quadratic polynomials concering related questions was then carried out by Sumi and Watanabe, in particular extending the results of \cite{LZ} for several more general settings (\cite{sumi_watanabe}, \cite{watanabe}).

We would like also to draw attention to the work of Comerford, in particular \cite{comerford1}, \cite{comerford2}.
These works exhibit some phenomena which occur in non-autonomous dynamics and which are unlikely for the autonomous one. For example, in \cite{comerford2} the author constructs a non- autonomous  polynomial system with positive Lebesgue Julia set and an  invariant line field on it.

\

In the present paper we are interested in the geometric properties  (mainly: Hausdorff dimension) of the harmonic measure on the Julia set of a random iteration of a quadratic polynomial and its dependence on the size of the region from which the parameter is chosen randomly. Below we explain the setting in more detail.

\

We consider non--autonomous compositions of quadratic polynomials $f_c=z^2+c$, where, at each step  $c$ is chosen 
from some bounded Borel set  $V\subset \mathbb C$ (e.g., the disc $\mathbb D(0,R)$).
Let us introduce the parameter space
$\Omega=V^\mathbb N.$
The space $\Omega$ is equipped with a natural left shift map $\sigma$.
Namely, for every bounded sequence $\underline c =(c_0,c_1, c_2,\dots)$ put
$$\sigma(\underline c)=(c_1, c_2, \dots).$$

\noindent Next, for every $\omega\in\Omega$, $\omega=(c_0,c_1,\dots)$ denote by  $f_\omega$ the map $f_{c_0}$.

\noindent Then the  non-autonomous composition  $f^n_\omega$ is given by the formula

$$f^n_\omega:=f_{c_{n-1}}\circ f_{c_{n-2}}\circ \dots \circ f_{c_0}.$$

\noindent The global dynamics can be described as a skew product
$F:\Omega\times \mathbb C\to \Omega\times\mathbb C,$
$$F(\omega, z)=\left (\sigma(\omega), f_\omega(z)\right ).$$

\noindent Then, for all $n\in\mathbb N$, we have that
$F^n(\omega,z)=\left (\sigma^n(\omega), f_\omega^n(z)\right )$.
So, every sequence $\omega\in\Omega$ determines a sequence of non-autonomous iterates
$\left (f_\omega^n\right )_{n\in\mathbb N}.$ 

\







Analogously to the autonomous case, we consider the following objects:

\begin{itemize}
\item[]{} \emph{escaping set}, or \emph{basin of infinity:} $$\mathcal A_\omega=\{z\in \mathbb C:f^n_\omega(z)\xrightarrow[n\to\infty]{} \infty\}$$
\item[]{} \emph{non-autonomous Julia set:}
$$J_\omega=\{z\in\mathbb C: \text{for every open set }~  U\ni z~\text{the family}~{f^n_\omega}_{|U}~  \text{is not normal.}\}$$
\item[]{} \emph{non- autonomous filled-in Julia set:}
$K_\omega=\mathbb C\setminus \mathcal A_\omega$.
\end{itemize}

\

The following proposition can be found in \cite{B} (Theorem 1). Its proof  is analogous to the autonomous case.

\begin{prop} Let $\omega\in \mathbb D(0,R)^\mathbb N$. Then 
$$J_\omega=\partial\mathcal A_\omega,$$
\end{prop}

Let us also note the following straightforward observations:

\begin{prop}\label{prop:est_julia}
For every $\omega\in\mathbb D(0,R)^{\mathbb N}$
$$J_{\sigma\omega}=f_\omega(J_\omega), \quad  \mathcal A_{\sigma\omega}=f_\omega(\mathcal A_{\omega})$$.
\end{prop}

Denote by   $\mu_\omega$ -- the equilibrium (harmonic) measure $\mu_\omega$  on $\partial\mathcal A_\omega=J_\omega$. (see Section~\ref{sec:green}).
The properties of this distribution  are the subject of this paper.

\

We  introduce random bounded systems of quadratic maps (see Definition~\ref{def:bounded_random}), where, in the consecutive steps, tha parameter $c$ is chosen according to the action of some ergodic automorphism on a probability space (see Section~\ref{sec:random_green} for the  definition).
The collection of equilibrium measure $(\mu_\omega)_{\omega\in\Omega}$ now becomes a random probability measure (see Section~\ref{sec:random_general}). 

We then introduce a class of bounded  random systems of quadratic maps, called \emph{independent, typically fast escaping}. This class of random systems was considered in \cite{LZ}. In particular, we proved in \cite{LZ} that for such systems the Julia set is almost surely totally disconnected.  A number of examples of random systems satisfying the above assumption is provided in \cite{LZ}. Perhaps the most surprising is the random system for which the parameter is chosen randomly from the main cardioid of the Mandelbrot set.

In Section~\ref{sec:formula} we prove

\begin{Theorem*} [Theorem~\ref{thm:dim_harm}]
Let $\mathcal S$ be a bounded random system of independent quadratic maps. Assume additionally that $\mathcal S$ is typically fast escaping.

\

Then for $\mathbb P$-a.e. $\omega\in\Omega$ 

\begin{equation}\label{eq:dim_max}
\dim_H(\mu_\omega)=\frac{\ln 2}{\chi}= \frac{\ln 2}{\ln 2+{\bf g}(0)}<1
\end{equation}
\end{Theorem*}
Here, ${\bf g}(0)$ is the value of the \emph{global Green's function}.

\

In Section~\ref{sec:dependence} we study random systems of quadratic maps, depending on a parameter which measures the ''scale of randomness''.

\

\begin{Theorem*}[see Theorems~ \ref{thm:continuous} and \ref{thm:asymp}]
Let $\mathbb P_R$ the product distribution of uniform 
distributions on $\mathbb D(0,R)$ in the product space $\mathbb D(0,R)^
{\mathbb N}$. Let ${\bf g}_R$ be the global Green's fuction for the random 
system of quadratic maps generated by this distribution. 

The generalized Green function ${\bf g}_R (0)$ depends continuously on $R$ and  satisfies
$\lim\limits_{R \rightarrow \infty} \frac{{\bf g}_{R}(0)}{\ln (R)} = \frac 1 2$.
Consequently,
$$\lim_{R\to\infty}{\dim_H(\mu^R_\omega)}\cdot {\ln R}=2\ln 2.$$
\end{Theorem*}
Here, $\mu_\omega^R$ denotes the harmonic measure on the set $J_\omega$, where $\omega\in \mathbb D(0,R)^{\mathbb N}$. For given $R$, the dimension  $\dim_H(\mu^R_\omega)$ is constant almost everywhere.

Further, we consider in Section~\ref{sec:dependence} random iterates of maps $f_c$ near a parameter $c_0$ which is outside the Mandelbrot set.

\begin{prop*}[see Proposition~ \ref{cor:green_harmonic3}, Proposition~\ref{prop:dim_analytic} and Proposition~\ref{prop:stronger}]
Let $\nu$ be a Borel probability measure on $\mathbb D(0,1)$. Let   $\mathbb P$ be the product distribution on $\Omega$ generated by $\nu$. Let $c_0\notin\mathcal M$ and let $\Lambda=\mathbb D(0,\delta_0)$, with $\delta_0>0$ sufficiently small. 
For $\omega\in \mathbb D(0,1)^\mathbb N put $ $f_{\lambda,\omega}= f_{c_0+\lambda \omega_0}$

 The function
$$\lambda\mapsto {\bf g}_\lambda(0)=\int_{\omega\in\Omega}g_{\lambda,\omega}(0)d\mathbb P(\omega)$$
is   harmonic  in $\Lambda$.


Consequently, the  functions $\delta\mapsto {\bf g}_\delta(0)$ and $[0,\delta_0)\ni \delta\mapsto \dim\mu_{\delta,\omega}$  are real analytic  in the segment $[0,\delta_0)$. 
If, additionally, the distribution $\nu$ is uniform, then the function
 $[0,\delta_0)\ni \delta\mapsto \dim\mu_{\delta,\omega}$ is constant.

\end{prop*}
The last statement  says that introducing the randomness has no influence on the dimension of the harmonic measure on a (typical) random Julia set, as long as the range of the random choice of  the parameter $c$ remains bounded by some positive constant $\delta_0$ (depending on the initial parameter $c_0$). The dimension is the same as  for the unperturbed autonomous iteration of $f_{c_0}$.

\medskip





\vskip 20pt

\noindent {\bf Notation.} For every $r>0$ denote $\mathbb D_r:=\mathbb D(0,r)$ and
$\mathbb D_{r}^* := \mathbb{C} \setminus \overline{\mathbb D}_{r}$.




We denote by $\dim_H$  Hausdorff dimension of a set. Similarly,  for a Borel probability measure $\mu$ we denote by $\dim_H(\mu)$  Hausdorff dimension of the measure $\mu$, i.e.,
$$\dim_H(\mu)=\inf\{\dim_H(A): A-\text { Borel sets of positive measure} ~  \mu\}$$


\subsection{Green's function and equilibrium measure.}\label{sec:green}
Our aim is to study the equilibrium (harmonic) measure on non-autonomus Julia sets $J_\omega$, where $\omega\in\mathbb D(0,R)^\mathbb N$.


We recall the proposition proved in  \cite{FS}, which we state in a slightly different form. See also \cite{LZ} for more details.

\begin{prop}\label{prop:def_green}

For every $\omega\in\mathbb D(0,R)^\mathbb N$ the following limit exists: $g_\omega:\mathcal A_\omega\to \mathbb R$:

\begin{equation}\label{eq:green}
g_\omega(z)= \lim_{n\to\infty} \frac{1}{2^n}\ln |f^n_\omega(z)|.
\end{equation}
 The function $z\mapsto g_\omega(z)$ is the Green function on $\mathcal A_\omega$ with pole at infinity. Each set $\mathcal A_\omega$ is regular for  Dirichlet problem.
 Putting $g_\omega\equiv 0$ on the complement of $\mathcal A_\omega$, $g_\omega$ extends continuously to the whole plane. 
 \end{prop}

 This is a generalization of a well- known formula for the autonomous case: for the map $f_c(z):=z^2+c$ and its  basin of infinity $\mathcal A_c$, Green function with a pole at infinity is given by:
 
$$g_c(z)= \lim_{n\to\infty} \frac{1}{2^n}\ln |f^n_c(z)|.$$

\begin{corollary}
We have 
\begin{equation}\label{eq:invariance}
g_{\sigma\omega}(f_\omega(z))=2 g_\omega(z).
\end{equation}
\end{corollary}
\begin{proof}
This follows directly from the formula \eqref{eq:green}, defining the Green function $g_\omega$.
\end{proof}

The measure $\mu_\omega=\frac{1}{2\pi}\Delta g_\omega$ (i.e., the harmonic measure for the domain $\mathcal A_\omega$, evaluated at infinity)
is supported on the Julia set $J_\omega$, and, since $g_{\sigma\omega}\circ f_\omega=2g_\omega$, it  satisfies

\begin{equation}\label{eq:maximal}
\mu_{\sigma \omega}(f_\omega(A))=2\mu_\omega(A)
\end{equation}

if $A$ is a Borel set such that $f_\omega$ is one-to-one on $A$.

We call this measure also the  maximal measure since, analogously to the autonomous case, $\mu_\omega$ is a fixed point of the (non- autonomous) transfer operator.
In order to explain this fact in more detail, let us start 
with  the following simple fact. Its proof is an easy calculation (see Propostion 5 in \cite{LZ}).
\begin{prop} \label{prop:log}
For every  $R>0$  there exists $R_0>0$ such that for every $\omega\in\mathbb D(0,R)^{\mathbb N}$ we have that
\begin{equation}\label{eq:1}
f_\omega(\mathbb D_{R_0}^*)\subset \mathbb D^*_{2R_0},
\end{equation}
and
\begin{equation}\label{eq:est_green}
|g_\omega(z)-\ln|z||<1
\end{equation}
for all $z\in\mathbb D_{R_0}^*$.
The inclusion \eqref{eq:1} implies that 
\begin{equation}\label{eq:2}
\mathbb D_{R_0}^*\subset \mathcal A_\omega.
\end{equation}
By the Maximum Principle,  the inequality \eqref{eq:est_green} implies that
\begin{equation}\label{eq:est_green_above}
g_\omega(z)\le \log R_0+1
\end{equation}
for all $z\in \mathbb D_{R_0}$ and all $\omega\in\mathbb D(0,R)^\mathbb N$.
\end{prop}

\

So, now, we choose $R_0$ according to the statement of Proposition~\ref{prop:log}. Denote by $C_{R_0}$ 
the space of continuous functions in $\overline{\mathbb D}_{R_0}$.

Consider  the family of  operators  $\mathcal L_\omega:C_{R_0}\to C_{R_0}$ defined as
$$\mathcal L_\omega(\varphi)(w)=\frac{1}{2}\sum_{z\in f_\omega^{-1}(w)}\varphi(z),$$ where $\varphi\in C_{R_0}$ and the preimages are counted with multiplicity.
Because of the choice of $R_0$, each operator $\mathcal L_\omega$ acts continuously  on the space $C_{R_0}$. 

\

Let $\mathcal L^*_\omega$ be the conjugate operator. Thus, $\mathcal L^*$ acts on the conjugate space $C_{R_0}^*$, i.e., the space of signed finite Borel measures on 
$\overline{\mathbb  D}_{R_0}$.

Notice that  $\mathcal L_\omega^*(\mu_{\sigma\omega})=\mu_\omega$.
Indeed, let $\varphi\in C_{R_0}$. Then

$$\mathcal L^*_\omega\mu_{\sigma\omega}(\varphi)=\mu_{\sigma\omega}(\mathcal L_\omega(\varphi))=\mu_\omega(\varphi)$$
by the equation \eqref{eq:maximal}.

The equation \eqref{eq:maximal} implies also  that the measure $\{\mu_\omega\}$ is ''fiberwise invariant":

\begin{equation}\label{eq:invariance2}
\mu_{\sigma\omega}=\mu_\omega\circ (f_\omega)^{-1}
\end{equation}

\

The following proposition says that the maximal measure can be expressed in terms of iterates  of the (non- autonomous) transfer operator. This is a non- autonomous version of the classical fact in the iteration of polynomials, which was first observed by Brolin, see \cite{brolin}, Theorem 16.1. 
The proof of the   non- autonomous version can be found in \cite{bruck}, Theorem 8.5.

\begin{prop}\label{prop:max_measure2}
For 
an arbitrary point $z_0\in \mathbb D_{R_0}^*$ the following holds:
$$\mu_{\omega}=\lim_{n\to\infty}\frac{1}{2^n}\sum_{y\in (f^n_{\omega})^{-1}(z_0)}\delta_y.$$
weakly, i.e., 
for every continuous function $\varphi$ defined on a neighbourhood of $J_{\omega}$,

$$\mu_{\omega}(\varphi)=\lim_{n\to\infty}\frac{1}{2^n}\sum_{y\in (f^n_{\omega})^{-1}(z_0)}\varphi(y).$$

In other words,

$$\mu_\omega=\lim_{n\to\infty}({\mathcal L^{*,n}_\omega})(\delta_{z_0}),$$
where we denoted by $({\mathcal L^{*,n}_\omega})$ the composition of the operators:
$$\mathcal L^{*,n}_\omega=\mathcal L^*_\omega\circ \mathcal L^*_{\sigma\omega}\circ\dots\circ\mathcal L^*_{\sigma^{n-1}\omega}$$

\end{prop}



\section{Introduction, part 2. Bounded random systems  of quadratic polynomials.}\label{sec:random_green}
We now introduce formally random iterations of quadratic polynomials. So, now we  consider non- autonomous iterations, where the choice of consecutive parameteres $c_n$ is governed by some 
probability distribution over some probability space. Below is the formal setting.

As in \cite{Arnold}, the randomness is modeled by a measure preserving dynamical system $(\Omega,\mathcal F,m,\sigma)$
where $(\Omega,\mathcal F,m)$ is a complete probability space and  $\sigma:\Omega\to\Omega$ -- an invertible measure preserving ergodic transformation.
To every $\omega\in\Omega$ 
we associate in a measurable way a  parameter $c=c(\omega)\in\mathbb C$, and consequently, the map denoted in the sequel by $f_\omega$ is defined as 

$$f_\omega:=f_{c(\omega)}.$$

This gives a natural way of defining random iterates: for $\omega\in\Omega$ put 
\begin{equation}\label{eq:random_iterate}
f^n_\omega=f_{\sigma^{n-1}\omega}\circ\dots\circ f_\omega.
\end{equation}

The basin of infinity $\mathcal A_\omega$, filled-in Julia set $K_\omega$ and the Julia set $J_\omega$ are defined in exactly the same way as in Section~\ref{sec:intro}.

In the sequel we shall always assume that the function $\omega\mapsto c(\omega)$ is bounded, so that there exists $R>0$ such that 
$c(\omega)\in\mathbb D(0,R)$ for every $\omega\in\Omega$. 
Then, for every $\omega\in\Omega$  the facts collected in Section~\ref{sec:intro} apply.  

Such a random system will be referred to as a \emph{bounded random system of quadratic maps}. So, the formal definition is the following.

\begin{definition}\label{def:bounded_random}
A bounded random system of quadratic maps   $\mathcal S$ is formed by a complete probability space 
$(\Omega, \mathcal F,\mathbb P)$ together with an ergodic $\mathbb P$-measurable measure preserving automorphism $\sigma:\Omega\to\Omega$, a positive constant $R>0$
and a measurable map $\omega\mapsto c(\omega)\in \mathbb D(0,R)$.
\end{definition}
\

In this random setting Proposition~\ref{prop:def_green} takes on the form:
\begin{prop}\label{prop:def_green2}

Let $\mathcal S$ be a bounded random system of quadratic maps. For every $\omega\in\Omega$ the following limit exists: $g_\omega:\mathcal A_\omega\to \mathbb R$:

\begin{equation}\label{eq:green2}
g_\omega(z)= \lim_{n\to\infty} \frac{1}{2^n}\ln |f^n_\omega(z)|.
\end{equation}
 The function $z\mapsto g_\omega(z)$ is the Green function on $\mathcal A_\omega$ with pole at infinity. Each set $\mathcal A_\omega$ is regular for  Dirichlet problem.
 Putting $g_\omega\equiv 0$ on the complement of $\mathcal A_\infty$, $g_\omega$ extends continuously to the whole plane. 
 With $z$ fixed, the function $\omega\mapsto g_\omega(z)$ is $\mathbb P$-- measurable.
 Given $R_0>0$, the functions $z\mapsto g_\omega(z)$ are uniformly bounded on $\mathbb D_{R_0}$.
 \end{prop}

\begin{remark}\label{rem:independ}
A natural example of a bounded random system of quadratic maps, which was  studied in \cite{LZ}, is the following.

Let $V\subset \mathbb D(0,R)$ be a Borel set, let $\nu$ be a Borel probability distribution on $V$. Put $\Omega=\Pi_{-\infty}^\infty V$, 
equipped with the (completed) product 
$\sigma$- algebra and the (completed) product probability $m$. Let $\sigma:\Omega\to\Omega$ be the left shift. Then $\sigma$ is an ergodic measure preserving automorphism.
The random iterates defined  in \eqref{eq:random_iterate} correspond in this case  to the random system, where, at each step, the parameter $c$
is chosen independently, according to the probability distribution $\nu$.
\end{remark}

\section{Random measures in Polish spaces}\label{sec:random_general}
We shall use a formalism of random sets and random measures. So, now we introduce necessary definitions and facts. We follow the presentation  elaborated in \cite{crauel}.
Let $(\Omega,\mathcal F, m)$ be a (complete) probability space, let $X$ be a Polish space.
Recall from \cite{crauel} that a function $\phi:\Omega\times X \to\R$, $\phi(\omega, z)=\phi_\omega(z)$, is called a random continuous function if
for every $\omega\in\Omega$ the function 
$$
X\ni z\longmapsto \phi_\omega(z)\in \R
$$ 
is continuous and bounded, and, in addition, for every $z\in X$ the function
$$
\Omega\ni\omega\longmapsto \phi(\om,z)\in\R
$$
is measurable. 
It then follows ( see, e.g., Lemma~1.1 in \cite{crauel}) that every random continuous function is measurable with respect to the product
$\sigma$--algebra $\mathcal F\otimes \mathcal B$, where  $\mathcal B$ is the Borel $\sigma$--algebra in $X$. Moreover, the map
$$
\Omega\ni\omega\longmapsto \|\phi_\omega\|_\infty\in\mathbb R
$$ 
is measurable and  $m-$ integrable.
The vector space of all real--valued random continuous functions is denoted by $C_b(\Omega\times X)$. Equipped with the norm 
$$
\|\phi\|:=\int_\Omega\|\phi_\omega\|_\infty dm(\om)
$$
it becomes a Banach space.

Put 
$$
\mathcal X:=\Omega\times X.
$$


Let 
$
\pi_1: \mathcal X \to \Omega
$
be the projection onto the first coordinate, i.e.,
$$
\pi_1(\om,z)=\om.
$$
Let $\mathcal M(\mathcal X)$ be the set of all non-negative probability measures defined on the $\sigma$-algebra $\mathcal F\otimes \mathcal B$,   and let $\mathcal {M}_m(X)$ be the set of all non-negative probability measures defined on the $\sigma$-algebra $\mathcal F\otimes \mathcal B$  that project onto $m$ under the map $\pi_1:X\to \Omega$, i.e. 
$$
\mathcal {M}_m=\big\{\mu\in\mathcal M (\mathcal X): \mu\circ\pi_1^{-1} =m\big\}.
$$
\begin{definition}
A map $\mu:\Omega\times \mathcal B\to[0,1]$, $(\omega, B)\longmapsto \mu_\omega(B)$, is called a random probability measure on $X$ if 
\begin{itemize}
\item{} For every set $B\in\mathcal B$ the function $\Om\ni\omega\longmapsto \mu_\omega(B)\in[0,1]$ is measurable,
\item{} For $m$-almost every $\omega\in\Omega$ the map $\mathcal B\ni B\mapsto \mu_\omega(B)\in[0,1]$ is a Borel probability measure.
\end{itemize}
\end{definition}
A random measure $\mu$ will be frequently denoted as $\{\mu_\omega\}_{\om\in\Om}$ or $\{\mu_\omega:\om\in\Om\}$.

The set $\mathcal M_m$ can be canonically identified with the collection of all random probability measures on $X$ as follows.

\begin{prop}[see Propositions 3.3 and 3.6 in\cite{crauel}]\label{prop:global_measure}
With the above notation, for every measure $\mu\in\mathcal M_m(X)$ there exists a unique random measure $\{\mu_\omega\}_{\om\in\Om}$ on $X$ such that 
$$
\int_{\Omega\times X}h(\omega,z)\,d\mu(\omega,z)
=\int_\Omega\left (\int_X h(\omega,z)\,d\mu_\omega(z)\right)dm(\omega)
$$
for every bounded measurable function $h:\Omega\times X\to \mathbb R$ { and for every measurable non-negative function $h:\Omega\times X\to \mathbb R$.}

Conversely, if $\{\mu_\omega\}_{\om\in\Om}$ is a random measure on $\mathcal X$, then for every bounded measurable  {or non- negative measurable} function 
$h:\mathcal X\to\mathbb R$ the function $\Om\ni\omega\longmapsto \int_{X}h(\omega,z)d\mu_\omega(z)$ is measurable, and the assignment
$$
\mathcal F\otimes \mathcal B\ni A\longmapsto\int_\Omega\int_{X}\1_A(\omega,z)d\mu_\omega(z) dm(\omega),
$$
defines a probability measure  (a ''global measure'') $\mu\in\mathcal M_m(\mathcal X)$.
\end{prop}

\


\

Another useful characterization of random probability measure on $X$ is the following (see Remark 3.20 in~ \cite{crauel}).
\begin{prop}\label{prop:rm}
A family of Borel probability measures on $X$:  $\mu_\omega$, $\omega\in\Omega$   is a random probability measure if and only if  for every bounded continuous function $\varphi:X\to\R$
the map
$$\omega\mapsto \mu_\omega(\varphi)$$
is measurable.
 
\end{prop}
\begin{definition}\label{def:random_closed}
Following Definition 2.1 in \cite{crauel} we say that a function $\Omega\ni\omega\mapsto C_\omega$, ascribing to each point $\omega\in\Omega$ a closed subset $C\subset X$
is called a random closed set if for each $z\in\C$ the function
$$\Omega\ni\omega\mapsto {\rm dist}(z,C_\omega)\in\R$$ is measurable. \end{definition}

Since the probability measure m on $\Omega$ is  
assumed to be complete, we can also give another characterization (see Proposition 2.4 in \cite{crauel}):

\begin{prop}
 Let $(\Omega, \mu)$ be a complete probability space. A function $\Omega\ni\omega\mapsto C_\omega\subset X$ is a random closed set if and only if 
 all sets $C_\omega$ are closed and, moreover,  the union (''graph'')
$$C:=\bigcup_{\omega\in\Omega}\{\omega\}\times C_\omega$$
is a measurable subset of $\Omega\times X$.
 
\end{prop}


We shall use yet another equivalent  characterization of random closed sets:

\begin{prop}[see Proposition 2.4 in \cite{crauel}]\label{prop:charact:random_closed}
$C$ is a random closed set in $X$ if and only if for every open set $U\subset X$ the set $\{\omega\in\Omega: U\cap C_\omega\neq\emptyset\}$ is measurable.
\end{prop}
\

\section{Random equilibrium measures}

In this section we consider a random bounded system $\mathcal S$ of quadratic polynomials, as defined in Definition ~\ref{def:bounded_random}. Let $R$ be a constant coming from this definition and let $R_0$ be assigned 
to $R$ according to Proposition~\ref{prop:log}.

We shall  interpret the objects defined in Section~ \ref{sec:random_green} in terms of the language introduced in Section~ \ref{sec:random_general}.

\begin{prop}\label{prop:random_harmonic_measure}
Let $\mathcal S$  be a bounded random system of quadratic maps. Then

\begin{enumerate}
 \item{} The filled-in random Julia sets  $K_\omega$ are random closed sets, 
 i.e., the collection $\{K_\omega\}_{\omega\in\Omega}$ satisfies the condition formulated in Definition ~\ref{def:random_closed}.
 \item{} The collection $g_\omega$ of random Green functions with poles at infinity is a random continuous function in $\mathbb D_{R_0}$.
 \item{} The collection $\{\mu_\omega\}_{\omega\in\Omega}$ of random equilibrium (harmonic) measures at infinity on $J_\omega$, defined by
 $$\mu_\omega=\frac{1}{2\pi}\Delta g_\omega$$
 is a random probability measure.
\end{enumerate}
\end{prop}

\begin{proof}
To prove item (1), let $U\subset \C$ be an arbitrary open set. We need to show that the set
\begin{equation}\label{eq:meas}
\{\omega\in\Omega:K_\omega\cap U\neq\emptyset\}
\end{equation}
is measurable.
Equivalently, we need to show measurability of the set 

$$\{\omega\in\Omega:K_\omega\cap U=\emptyset\}=\{\omega\in\Omega:U\subset\mathcal A_\omega\}.$$
Let $(V_n)_{n\in\mathbb N}$ be an ascending sequence of open sets, relatively compact in $U$, such that $\bigcup_{n\in\mathbb N} V_n=U$.

Then $U\subset \mathcal A_\omega$ if and only if $V_n\subset\mathcal A_\omega$ for all $n\in\mathbb N$  and we are left to show that all the sets
\begin{equation}\label{eq:v_n}
\{\omega\in\Omega: \overline{V}_n\subset \mathcal A_\omega\}
\end{equation}
are measurable.

Since $\overline V_n$ is compact, the condition expressed in \eqref{eq:v_n} can be written in terms of Green function:
\begin{equation}\label{eq:dense}
\{\omega\in\Omega: \overline{V}_n\subset \mathcal A_\omega\}=\{\omega\in\Omega:\inf_{z\in {\overline V}_n} g_\omega(z)>0\}
\end{equation}
Let $\{z_k\}_{k\in \mathbb N}$ be a dense set in $U$. The condition expressed in \eqref{eq:dense} can be written equivalently as
\begin{equation}\label{eq:dense2}
\{\omega\in\Omega:\inf_{z_k\in {\overline V}_n} g_\omega(z_k)>0\}
\end{equation}
Since for each $k$ the function $\omega\mapsto g_\omega(z_k)$ is measurable, the set defined in \eqref{eq:dense2} is measurable, 
and, therefore the sets 
defined in \eqref{eq:dense}, 
 \eqref{eq:v_n}
and \eqref{eq:meas}  are measurable as well.

Item (2) follows directly from Proposition ~\ref{prop:def_green2} and from the definition of a random continuous function.

To prove item (3) we use the characterization of random probability measures formulated in Proposition~\ref{prop:rm}.
To continue the proof we need two lemmas.
\begin{lemma}\label{lem:lemma1}
 Denote by  $\underline c=(c_0,c_1,\dots, c_{n-1})$  a sequence of parameters in $\mathbb D(0,R)$ i.e. $\underline c \in \mathbb D (0,R)^n$.  Let $R_0$ be the value assigned to $R$
 in Proposition~\ref{prop:log}.

\noindent Let $\varphi$ be a continuous function in $\overline{\mathbb D}_{R_0}$.
Let $z_0\in \overline{\mathbb D}_{R_0}$.  Then the function
\begin{equation}\label{eq:continuous}
\left (\mathbb D(0,R)\right )^n\ni \underline c\mapsto \sum_{v\in (f^n_{\underline c})^{-1}(z_0)}\varphi(v)
\end{equation}
(where the preimages are counted with multiplicity)
is continuous.
\end{lemma}
\begin{proof}
The points $v\in (f^n_{\underline c})^{-1}(z_0)$ are roots of the polynomial $z\mapsto f^n_{\underline c}(z)-z_0$. 
The coeffcients of this polynomial depend also polynomially on
the coefficients of $\underline c=(c_0,\dots c_{k-1}).$ Continuity of the function ~\eqref{eq:continuous} thus 
follows immediately from the fact that roots of a complex polynomial (counted 
with multiplicity) depend continuously on the coefficients.
 
\end{proof}

\begin{lemma}\label{lem:lemma2} Let $\varphi$  be a continuous function in $\overline{\mathbb D}_{R_0}$.
Let $z_0\in \overline{\mathbb D}_{R_0}$. For each $n\in\mathbb N$  the function
$$\Omega\ni\omega \mapsto \sum_{v\in (f^n_{\omega})^{-1}(z_0)}\varphi(v)$$
is measurable.
\end{lemma}
\begin{proof}
It is enough to observe that the function

$$\omega\mapsto\underline c(\omega)= (c_0(\omega),c_1(\omega)\dots c_{n-1}(\omega))$$
is measurable, whereas
the function
$$\underline c\mapsto\sum_{v\in (f^n_{\underline c})^{-1}(z_0)}\varphi(v)$$
is continuous, by Lemma~\ref{lem:lemma1}.
\end{proof}
We can now conclude the proof of Proposition~\ref{prop:random_harmonic_measure}.
Indeed, defining by $\mu_{n,\omega}$ the measure equidistributed over the set $\{v\in (f^n_{\omega})^{-1}(z_0)\}$ 
(counted with multiplicity), we conclude from Lemma~\ref{lem:lemma2}
that the functions 
$$\omega\mapsto \mu_{n,\omega}(\varphi)$$
are measurable.
Since for every $\omega$ the sequence $\mu_{n,\omega}(\varphi)$ converges to $\mu_\omega(\varphi)$ (see Proposition~\ref{prop:max_measure2}), we conclude that
the limit function $\omega\mapsto \mu_\omega(\varphi)$ is measurable.

\end{proof}

Since, by Proposition ~\ref{prop:random_harmonic_measure} $\{\mu_\omega\}_{\omega\in\Omega}$ is a random probability measure, 
we can consider the global measure $\mu$ on the 
product space $\Omega\times \C$, as in Proposition~\ref{prop:global_measure}, and the global map (skew product)
$$F:\Omega\times \mathbb C\to \Omega\times \mathbb C$$ determined by the formula
 
 \begin{equation}\label{eq:skew}
 F(\omega, z)=\left (\sigma\omega, f_\omega(z)\right ).
 \end{equation}

\begin{prop}\label{prop:ergodic_mean}
 The global map $F$ is $\mu$--measurable and  the measure $\mu$ is invariant under the map $F$.

\end{prop}
\begin{proof}
 First, we prove measurability of the map $F$. It is enough to prove that for every  set of the form $C\times B$ where $C \in \mathcal{F}$ and $B$ is a compact subset of $\C$, its preimage
 $F^{-1}(C\times B)$ is measurable.
 We have, by definition of the map $F$,
 \begin{equation}\label{eq:measurability}
 F^{-1}(C\times B)=\bigcup_{\omega\in\sigma^{-1}(C)}\{\omega\}\times f^{-1}_\omega(B).
 \end{equation}
 By the characterization of a random closed set, to prove measurability of \eqref{eq:measurability}, we need to prove that the collection

 $$B_\omega=\begin{cases}f^{-1}_\omega(B)\quad\text{for}\quad\omega\in \sigma^{-1}(C),\\
 \emptyset\quad\text{for}\quad \omega\notin \sigma^{-1}(C)
          \end{cases}
            $$ 
  is a random closed set.          
 To prove measurability of this set we use the characterization given in Proposition~\ref{prop:charact:random_closed}. 
 
 Let $U\subset \C$ be an open set. We need to check that the set 
 $$\{\omega\in\sigma^{-1}(C):U\cap f_\omega^{-1}(B)=\emptyset\}=\sigma^{-1}(C)\cap \{\omega\in\Omega:U\cap f_\omega^{-1}(B)=\emptyset\}$$
 is measurable. Clearly, it is sufficient to check measurability of the set $$\{\omega\in\Omega:U\cap f_\omega^{-1}(B)=\emptyset\}.$$
 Since by definition the function $\omega \mapsto c (\omega) \in \mathbb{D} (0,R)$ is measurable, it is in fact sufficient to verify Borel measurability of the set 
 $$
 \{c \in \mathbb{D} (0,R) : U \cap f_c^{-1}(B)=\emptyset \}
 $$
 or, equivalently,, Borel  measurability of  the set
 $$
 \mathcal C:=\{c \in \mathbb{D} (0,R) : U \cap f_c^{-1}(B) \neq \emptyset \}.
 $$
Since $U$ is open the above set is also open, and thus measurable.
 

 Next, we check that $F$ preserves the measure $\mu$. So, we need to show that for every $\mu$-- measurable set $A$ we have $\mu(F^{-1}(A))=\mu(A)$.
 
It is enough to check the above  equality for every set of the form $A:=C\times B$, where $C \in \mathcal F$ and $B$ is a Borel measurable subset of $\mathbb C$, since these sets generate 
the product $\sigma$--algebra  $\mathcal F\otimes\mathcal B$.
 We have:
 
 $$F^{-1}(C\times B)=\bigcup_{\omega\in \sigma^{-1}(C)}\{\omega\}\times f^{-1}_\omega(B).$$
 
 Thus,
 
 \begin{align*}&\mu(F^{-1}(C\times B))=\int_{\omega\in\sigma^{-1}(C)}\mu_\omega(f^{-1}_\omega(B))\mathbb{P}(\omega)\\&=
 \int_{\omega\in\sigma^{-1}(C)}\mu_{\sigma\omega}(B)\mathbb{P}(\omega) =  \int_{\omega \in C}\mu_{\omega}(B)\mathbb{P}(\omega) = \mu (C \times B)
 \end{align*}
 \end{proof}

\begin{remark}
It is natural to ask about ergodicity of the above skew product. It will be proved in the next section, under some additional natural assumption (Proposition~\ref{prop:global_ergodic}).
\end{remark}

\section{Random systems of quadratic  maps. Dimension of the maximal measure}\label{sec:formula}



Let us recall that, as noted in Proposition~\ref{prop:def_green2}, if $\mathcal S$ is a bounded random quadratic system, then the function 
the function $\omega\mapsto g_\omega(0)$ is $\mathbb P$-- measurable and bounded. Therefore, the integral 

$${\bf g}(0):=\int_\Omega g_\omega(0)d\mathbb P(\omega)$$ is well- defined. This value will be referred to as the global Green function.

We now define the random Lyapunov exponent:  

\begin{equation}\label{eq:lyap_random}
\chi_\omega=\int_{J_\omega}\ln |f'_\omega|(z)d\mu_\omega(z)=\int_{J_\omega}\ln|2z|d\mu_\omega(z),
\end{equation}
and, next, the global Lyapunov exponent:
\begin{equation}
\chi=\int_\Omega\left (\int_{J_\omega}\ln |f'_\omega|(z)d\mu_\omega(z)\right )d\mathbb P(\omega)=\int_{\Omega\times\mathbb C}\ln|f'_\omega|(z)d\mu(\omega,z),
\end{equation}
where $\mu$ is the global measure on $\Omega\times\mathbb C$ with fiber measures $\mu_\omega$. 
\begin{prop}\label{prop:exponent}
The global Lyapunov exponent can be calculated in terms of the global Green function:

$$\chi=\ln 2+{\bf g}(0).$$
\end{prop}
\begin{proof} The proof does not differ much from the autonomous case. We provide the details for completeness.
$$\chi_\omega=\int_{J_\omega}\ln |f'_\omega|(z)d\mu_\omega(z)=\int_{J_\omega}\ln|2z|d\mu_\omega(z)=\ln 2+\int_{J_\omega}\ln|z|d\mu_\omega(z)$$

So, we need to calculate the integral

\begin{equation}\label{eq:potential_at zero}
\int_{J_\omega}\log|z|d\mu_\omega(z)
\end{equation}
Recall that for a Borel measure $\eta$  in $\mathbb C$, the formula
$p_\mu(w)=\int \ln|z-w|d\eta(z)$
defines a subharmonic function in $\mathbb C$, and $\Delta p_\eta=2\pi \eta$.

In the case we consider, $\eta=\mu_\omega$ is the equilibrium measure (i.e. the measure maximizing the  integral

$$\iint\ln|x-y|d\nu(x)d\nu(y),$$
where $\nu$ runs over all Borel probability measures on $J_\omega$).

Since each set $J_\omega$ is regular for Dirichlet problem, the function\\
$$x\mapsto\int\ln|x-y|d\mu_\omega(y)$$
is constant on $J_\omega$. 

\

So, for each $J_\omega$ we have:

\begin{enumerate}
\item{}$p_{\mu_\omega}$ is constant on $J_\omega$, say $p_{\mu_\omega}\equiv P$ on $J_\omega$.

\item{}$g_{\omega}\equiv 0$ on $J_\omega$.
\end{enumerate}

\

Since $\Delta p_{\mu_\omega}=\Delta g_\omega= 2\pi \mu_\omega$, we conclude that
$p_{\mu_\omega}=g_\omega+h$ where $h$ is some harmonic function in $\mathbb C$.

Observe that 
$h(z)\to 0$ as $z\to \infty$ (because both $g_\omega$ and $p_{\mu_\omega}$ are of the form $\log|z|+o(1)$ as $z\to\infty$),
and  $h(z)\equiv P$ on $J_\omega$. Since $h$ is harmonic, this implies that $P=0$, and consequently
$p_{\mu_\omega}= g_\omega.$

In particular, $p_{\mu_\omega}(0)=g_\omega(0)$,  which gives the value of the integral \eqref{eq:potential_at zero}, and shows that

$$\chi_\omega=\ln 2+ g_\omega(0).$$

The formula for the global  Lyapunov exponent follows by a direct  integration.
\end{proof}

\

\begin{remark}\label{rem1}
Recall that  $\mathcal S$ is a bounded random quadratic system, so there exists $R>0$ such that  $c(\omega)\in\mathbb D(0,R)$ for every $\omega\in\Omega$. 
Using Proposition~\ref{prop:log}, we see that for every $\omega$ the set $K_\omega$ is contained in $\mathbb D_{R_0}$,  where $R_0$ is the value coming from Proposition~\ref{prop:log}.
Thus,
\begin{equation}\label{eq:est_exponent}
\chi_\omega< \ln 2+\log R_0.
\end{equation}
This follows from the fact that $\log|f'_\omega(z)|=\log2+\log|z|<\log2+\log R_0$ for every $z\in J_\omega$.
\end{remark}

In order to give the exact formula for the dimension of the measure $\mu_\omega$ we need to restrict our considerations to a natural subclass of bounded random system of quadratic maps.

\begin{definition}\label{def:random_independent}
 A bounded random system of quadratic maps $\mathcal S$ is called a \emph{system of independent random quadratic maps} if the probability space $\Omega$ is the infinite product
 $V^{\mathbb N}$ for some Borel bounded set $V\subset \mathbb C$ and $\mathbb P$ is the (completed) product distribution generated by  some Borel distribution $\nu$ on $V$.
 \end{definition}

We need one more definition:

\begin{definition}[Definition 2.2 in \cite {LZ}.]\label{typ_fast_esc}
Let $\mathcal S$ be a bounded  system of independent random quadratic maps, with $V\subset\mathbb D(0,R)$. The system $\mathcal S$ 
is called \emph{typically fast escaping}
if there  exists $\gamma>0$ such that 
\begin{equation}\label{eq: fast_escaping}
\mathbb P\left(\left \{\omega\in\Omega: g_\omega(0)<\frac{1}{2^k}\right \}\right )<e^{-\gamma k}
\end{equation}
\end{definition}

Consequently,  for fast escaping system for almost every $\omega\in\Omega$ the critical point $0$ escapes to infinity under the iterates $f^n_\omega$. 
It was proved in \cite{LZ}
that for typically fast escaping systems the Julia set $J_\omega$ is almost surely totally disconnected. The same work also provides proof that the typically fast escaping property holds for many natural systems $\Omega$, such as products of disks of radius larger than $\frac{1}{4}$ or the main cardioid of the Mandelbrot set, each equipped with the uniform distribution.

\

In this section we prove the following theorem. 
 
\begin{theorem}\label{thm:dim_harm}
Let $\mathcal S$ be a bounded random system of independent quadratic maps. Assume additionally that $\mathcal S$ is typically fast escaping.

\

Then for $\mathbb P$-a.e. $\omega\in\Omega$ 

\begin{equation}\label{eq:dim_max}
\dim_H(\mu_\omega)=\frac{\log 2}{\chi}= \frac{\ln 2}{\ln 2+{\bf g}(0)}<1
\end{equation}
\end{theorem}


\begin{remark}
Using our notation, it is not difficult  to observe that  for every $\omega\in\mathbb D(0,R)^\mathbb N$ 
the Green function $g_\omega$ is H\"older continuous with any H\"older exponent smaller than 
$$\alpha_0=\frac{\ln 2}{\ln 2+\ln R_0}.$$






Now,   $\alpha_0$ -- H\"older continuity of the Green function allows to conclude directly  that the  Hausdorff dimension of the equilibrium measure satisfies
 $\dim_H(\mu_\omega)\ge\alpha_0$
(see \cite{ransford2}, Lemma 3.5).

 For $\mathbb P$-- typical $\omega$, Theorem~ \ref{thm:dim_harm}  gives, 
using also
Remark~\ref{rem1}, the common value of the dimension of $\mu_\omega$:

$$ \dim_H(\mu_\omega)=\frac{\ln 2}{\chi}> \frac{\ln 2}{\ln 2+\ln R_0}=\alpha_0.$$
\end{remark}

\

\

Before starting the proof of Theorem~\ref{thm:dim_harm}, we need to refer to more definitions and results from  \cite{LZ}.

\

Following \cite{LZ}, we introduce a constant $\tilde R_0\in (R_0,R_0^2-R)$, say, $ \tilde R_0=\frac{2 R_0-R}{2}$. 
Then for every $\omega\in\mathbb D(0,R)^\mathbb N$, $f_\omega^{-1}(\mathbb D_{\tilde R_0})\subset \mathbb D_{R_0}$.

\

For our application explained below  it is only important that $\tilde R_0>R_0$.

\

Put $\tilde D=\mathbb D_{\tilde R_0}$ and 
Put $D=\mathbb D_{R_0}$.
\begin{theorem}[see Theorem 11 in \cite{LZ}] \label{thm:main_step}
Let $V$ be a bounded Borel subset of $\mathbb C$, $V\subset \mathbb D(0,R)$. Let $R_0$ be the 
value assigned to $R$ as in Proposition~\ref{prop:log}.
Let $\nu$ be a Borel probability measure on $V$ and let $\mathbb P$ be the product distribution on $\Omega= V^
\mathbb N$, generated by $\nu$. 
If the system $\mathcal S$ is \emph{typically fast escaping}, i.e., if 
\eqref{eq: fast_escaping} holds, then there exist $N\in\mathbb N$ such that for $\mathbb P$--almost every $\omega\in\Omega$ the following holds:

(*) There is a sequence $k_n=k_n(\omega)\to\infty$ such that 
for  each component $D^{k_n}_j(\omega)$ of the set $\tilde D^{k_n}(\omega)=(f^{k_n}_\omega)^{-1}(\tilde D)$ the degree of  the map
$$f^{k_n}_\omega:\tilde D^{k_n}_j(\omega)\to \tilde D$$
is at most $N$.
\end{theorem}

Actually,  the proof of Theorem 11 in \cite{LZ} says more: 

\begin{prop}[Complement to Theorem~ \ref{thm:main_step}]\label{prop:complement}
 There exists $K\in\mathbb N$ such that for  $\mathbb P$--almost every $\omega\in\Omega$ the following holds:

 (**)For  each component $ \tilde D^{k_n}_j(\omega)$ of the set $ (f^{k_n}_\omega)^{-1}( D)$ the degree of  the map
$$f^{k_n-K}_\omega:\tilde D^{k_n}_j(\omega)\to f^{k_n-K}_\omega( \tilde D^{k_n}_j(\omega))$$
is equal to one.

Moreover, for $\mathbb P$- a.e. $\omega$ the sequence $k_n=k_n(\omega)$ may be chosen in such a way that
\begin{equation}\label{eq:dense4}
 \lim_{n\to\infty}\frac{k_{n+1}-k_n}{k_n}=0
\end{equation}
 \end{prop}
We refer to the proof of Theorem 11 in \cite{LZ}. The sequence $k_n$ corresponds to consecutive visits of a trajectory to some positive measure set, under some ergodic automorphism. The estimate \eqref{eq:dense} is then an easy  consequence of the ergodic theorem.

Before proving Theorem ~\ref{thm:dim_harm} we check the announced  ergodicity of the global map $F$. The proof of Proposition~\ref{prop:global_ergodic}  is based on Proposition~\ref{prop:complement} together with  rather standard reasoning.


\begin{prop}\label{prop:global_ergodic}
 Let $\mathcal S$ be a  typically fast escaping (see Definition \ref{typ_fast_esc}) bounded system of independent random quadratic maps. 
 Then the global map $F$ defined in the formula \eqref{eq:skew} is ergodic with respect to the global measure $\mu$.
\end{prop}

\begin{proof}[Proof of Proposition \ref{prop:global_ergodic}]
Assume the contrary, that is that $F$ is not ergodic with respect to $\mu$. Then there exist two $F$- invariant measurable  disjoint sets $A$, $B$ of positive measure $\mu$. Let us denote by $A_\omega$ the preimage of the projection $\pi^{-1}(\omega)$ where $\pi : A \to \Omega$. We define the set $B_{\omega}$ analogously. Then for each $\omega \in \Omega$ the sets $A_{\omega}, B_{\omega}$ are Borel measurable and the functions $\omega \mapsto \mu_{\omega}(A_{\omega}), \omega \mapsto \mu_{\omega}(B_{\omega})$ are $\mathbb{P}$-measurable (refer for example to \cite{crauel}, Corollary 3.4). 
 
 Invariance of $A,B$ under $F$ implies that 
 $$
 f^{-1}_{\omega} (A_{\sigma \omega}) = A_{\omega}, \, f^{-1}_{ \omega} (B_{\sigma \omega}) = B_{\omega}.
 $$
Moreover invariance of $\mu$ implies that for $\mathbb{P}$ almost every $\omega \in \Omega$ we have
$$
\mu_{\omega}(f^{-1}_{ \omega}(A_{\sigma \omega})) = \mu_{\sigma \omega} (A_{\sigma \omega})$$ and $$\mu_{\omega}(f^{-1}_{ \omega}(B_{\sigma \omega})) = \mu_{\sigma \omega} (B_{\sigma \omega}).
$$
From the above we have that the functions
$$
\varphi (\omega) = \mu_{\omega} (A_{\omega})\quad\text{and}\quad
   \psi (\omega) = \mu_{\omega} (B_{\omega})
$$
are $\sigma$-invariant, and thus, by ergodicity of $\sigma$, constant $\mathbb{P}$-almost everywhere.
 
 Consider two Borel probability measures $\mu_A$ and $\mu_B$ , defined by their fiber measures as follows:
 
 $$\mu_{A,\omega}(F)=\frac{\mu_\omega(A_\omega\cap F)}{\mu_\omega(A_\omega)},$$
and, analogously
 $$\mu_{B,\omega}(F)=\frac{\mu_\omega(B_\omega\cap F)}{\mu_\omega(B_\omega)},$$
 where $F$ is a Borel set.
 Both $(\mu_{A,\omega})_{\omega \in \Omega}, (\mu_{B,\omega})_{\omega \in \Omega}$ are random probability measures. Indeed, to confirm that for $(\mu_{A,\omega})_{\omega \in \Omega}$ we need to check that the function $$\omega \mapsto \mu_{\omega} (A_{\omega} \cap F)$$ is $\mathbb{P}$-measurable. This follows again from \cite{crauel}, Corollary 3.4., applied to the set $(\Omega \times F) \cap A$. 
 
 Note that if $S \subset \mathbb{C}$ is a Borel measurable set on which $f_{\omega}$ is injective then from the fact that $\varphi (\omega)$ is constant almost everywhere and from \eqref{eq:maximal} we get
 $$
 \mu_{A,\sigma\omega}(f_\omega(S))=2\cdot \frac{\mu_\omega(A_\omega)}{\mu_{\sigma\omega}(A_{\sigma\omega})}\mu_{A,\omega}(S) = 2 \mu_{A, \omega} (S).
 $$
 Similarly if $f^k_{\omega}$ is injective on $S$ then
 $$
 \mu_{A, \sigma^k \omega} (f^k_{\omega}(S)) = 2^k \mu_{A, \omega} (S).
 $$
 Clearly the above equalities also hold if we replace the set $A$ with $B$. Finally if $S,S'$ are topological disks and $f^k_{\omega}: S \mapsto S'$ is a proper map of some degree $d \geqslant 1$ then
 $$
 \mu_{A, \sigma^k \omega} (S') \leqslant 2^k \mu_{A, \omega} (S)\quad\text{and}\quad \mu_{B, \sigma^k \omega} (S') \leqslant 2^k \mu_{B, \omega} (S)
 $$
To prove the above note first that the map $f^k_{\omega} : S \mapsto S'$ is a branched covering map of degree $d$. Since the measure $\mu_{\omega}$ has no atoms, one can find an arc $L \subset S'$ with one endpoint inside $S'$ and the other in $\partial S'$, such that $L$ contains all critical values and $\mu_{\sigma^k \omega}(L) = 0$. The set $S' \setminus L$ is simply connected. Choose a branch of $(f^k_{\omega})^{-1}$ defined on $S' \setminus L$ and mapping this set onto its image  $S'' \subset S$. Then $f^k_{\omega} : S'' \mapsto S' \setminus L$ is injective, and thus indeed we get
$$
\mu_{A, \sigma^k \omega} (S') = \mu_{A, \sigma^k \omega} (S' \setminus L) = 2^k \mu_{A, \omega} (S'') \leqslant 2^k \mu_{A, \omega} (S),
$$
and,  obviously, the same holds for the measures $\mu_{B, \omega}$.

\

Now denote as before $D = D_{R_0}$. Let $z \in \mathcal{J}_{\omega}$ and let $S_k (z)$ be the connected component of $(f^{k}_{\omega})^{-1}(D)$ containing $z$. Choose an arbitrary $\omega \in \Omega$ for which the statement of Proposition \ref{prop:complement} holds and let $k_n$ be the sequence from this Proposition which is assigned to $\omega$. Let $z \in \mathcal{J}_{\omega}$. Then the map $$f^{k_n - K}_{\omega} : S_{k_n} (z) \mapsto f^{k_n - K}_{\omega} (S_{k_n}(z))$$ is injective, and therefore we have
$$
\mu_{A , \omega } (S_{k_n}(z)) = \frac{1}{2^{k_n -K}} \mu_{A, \sigma^{k_n - K}\omega} (f^{k_n - K}_{\omega} (S_{k_n}(z))) \leqslant \frac{1}{2^{k_n - K}}.
$$
For the measure $\mu_{B, \omega} (S_{k_n}(z))$ we will use the following lower estimate
$$
\mu_{B,\omega}(S_{k_n}(z)) \geqslant \frac{1}{2^{k_n}} \mu_{\sigma^{k_n}\omega , B} (D) = \frac{1}{2^{k_n}}. 
$$
Combining the above we obtain
\begin{equation}\label{eq:qwerty}
\frac{\mu_{A,\omega}(S_{k_n}(z))}{\mu_{B,\omega}(S_{k_n}(z))} \leqslant 2^K.
\end{equation}
We shall show that this implies absolute continuity of the measure $\mu_{A,\omega}$ with respect to $\mu_{B,\omega}$. 

 Consider a set $C \subset \mathcal{J}_{\omega}$, such that $\mu_{B, \omega} (C) = 0$. Set $S_n (C) = \bigcup_{z \in C} S_n (z)$ and now since for almost every $\omega$ the set $\mathcal{J}_{\omega}$ is totally disconnected, we have $$C = \bigcap_{n}S_n (C). $$
This yields
$$
\mu_{A,\omega}(C) = \lim\limits_{n \rightarrow \infty} \mu_{A, \omega} (S_n (C))
$$
and analogously
$$
\mu_{B,\omega}(C) = \lim\limits_{n \rightarrow \infty} \mu_{B, \omega} (S_n (C)).
$$
Thus for every $\delta > 0 $ there exists $n$ such that $$\mu_{B, \omega} (S_{k_n}(C)) < \delta.$$ A straightforward application of \eqref{eq:qwerty} implies that we also have
$$
\mu_{A,\omega}(S_{k_n} (C)) < \delta 2^K
$$
and thus also
$$
\mu_{A,\omega}(C) < \delta 2^K.
$$
Since $\delta$ is arbitrarily small, we have $\mu_{A, \omega} (C) = 0$ and thus indeed, $\mu_{A,\omega}$ is absolutely continous with respect to $\mu_{B, \omega}$, which is a contradiction, since for almost all $\omega$ we have $A_{\omega} \cap B_{\omega} = \emptyset$.
\end{proof}

\begin{proof}[Proof of Theorem \ref{thm:dim_harm}]

Choose $\omega\in\Omega$ for which satisfying the assertion of Proposition~\ref{prop:complement} holds.
Let $z\in J_\omega$. Consider the sequence $D^{k_n}_j(z)$ of connected components of $(f_\omega^{k_n})^{-1}(D)$, containing the point $z$.  Every such component $D^{k_n}_j(z)$
is mapped by 
$f_\omega^{k_n-K}$ univalently onto its image $B$, which is some component in the collection $\{D^K_i\}$ of connected components of $f^{-K}(D)$.

The following lemma is then a straightforward consequence of the Koebe distortion theorem.

\begin{lemma}\label{lem:balls}
There exists $C>0$ such that the following holds.  Choose an arbitrary  $\omega\in\Omega$ satisfying  Proposition~\ref{prop:complement} and $z\in J_\omega$.
Let $k_n=k_n(\omega)$ be the sequence 
described in the assertion of Proposition ~\ref{prop:complement}. Then 

$$\mathbb D \left (z,\frac{1}{C} \cdot \frac{1}{|(f^{k_n-K}_\omega)'(z)|}\right )\subset D^{k_n}_j(z)\subset \mathbb D \left (z,C\cdot \frac{1}{|(f^{k_n-K}_\omega)'(z)|}\right )$$
 
\end{lemma}

\

\


\




Now,  note that the measure $\mu_\omega$ of such component $D^{k_n}_j(z)$ satisfies $$\mu_{\omega}(D^{k_n}_j(z)) = \frac{1}{2^{k_n-K}}\cdot \mu_{\sigma^{k_n-K}\omega}(B)\le\frac{1}{2^{k_n-K}}.$$ This is a consequence of the formula \eqref{eq:maximal}.

So, denoting $r_n=r_n(z)=  \frac{1}{C}\frac{1}{|(f^{k_n-K}_\omega)'(z)|}$  we conclude, using Proposition~\ref{prop:global_ergodic} applied to the function $(\omega,z)\mapsto 
\ln|f'_\omega(z)|$,
that for $\mathbb P$-a.e. $\omega\in\Omega$, for $\mu_\omega$ a.e. $z\in J_\omega$:

$$\frac{-\ln\mu_\omega (\mathbb D(z,r_n))}{-\ln r_n}\ge\frac{(k_n-K)\ln 2}{\ln|(f^{k_n-K}_\omega)'(z)|-\ln C}\xrightarrow{n\to\infty}\frac{\ln 2}{\chi},$$
so,
$$\liminf_{n\to\infty}\frac{-\ln\mu_\omega (\mathbb D(z,r_n))}{-\ln r_n}\ge\frac{\ln 2}{\chi}$$
and, similarly,
denoting $R_n=R_n(z)=  {C}\frac{1}{|(f^{k_n-K}_\omega)'(z)|}$, we conclude that
$$\limsup_{n\to\infty}\frac{-\ln\mu_\omega (\mathbb D(z,R_n))}{-\ln R_n}\le\frac{\ln 2}{\chi}$$

Now, let $r>0$ be an arbitrary radius in $(0, r_1)$. Take the largest $n$ such that $r<r_n$. Then
$$r_{n+1}\le r<r_n$$ and
\begin{equation}\label{eq:compare_ratio}
\frac{-\ln\mu_\omega(\mathbb D(z,r)}{-\ln r}\ge\frac{-\ln\mu_\omega \mathbb D(z,r_n)}{-\ln r}
=\frac{-\ln\mu_\omega \mathbb D(z,r_n)}{-\ln r_n}\cdot\frac{-\ln r_n}{-\ln r}
\end{equation}
Next, notice that
\begin{equation}\label{eq:dense3}
 \frac{-\ln r_n}{-\ln r}\ge\frac{-\ln r_n}{-\ln r_{n+1}}=\frac{\ln|(f^{k_n-K}_\omega)'(z)|-\ln C}{\ln|(f^{k_{n+1}-K}_\omega)'(z)|-\ln C}=\frac{k_n}{k_{n+1}}\cdot 
\frac{\frac{1}{k_n} (\ln|(f^{k_{n}-K}_\omega)'(z)|-\ln C)}{\frac{1}{k_{n+1}} (\ln|(f^{k_{n+1}-K}_\omega)'(z)|-\ln C)}
\end{equation}
Since, by Proposition ~\ref{prop:complement}, $\frac{k_n}{k_{n+1}}\to 1$ as $n\to\infty$, combining \eqref{eq:compare_ratio} and \eqref{eq:dense3}, we conclude that

$$\liminf_{r\to 0}\frac{-\ln \mu_\omega(\mathbb D(z,r))}{-\ln r}\ge \frac{\log 2}{\chi}.$$
An analogous reasoning with $r_n$ replaced by $R_n$ gives that 
$$\limsup_{r\to 0}\frac{-\ln \mu_\omega(\mathbb D(z,r))}{-\ln r}\le \frac{\log 2}{\chi}.$$

We conclude that for $\mathbb P$ -a.e. $\omega\in\Omega$ and  
and for a.e. $z$ with respect to $\mu_\omega$ 
$\lim_{r\to 0}\frac{\ln\mu_\omega(\mathbb D(z,r))}{\ln r}$
exists, and 
and it is  equal to  $\frac{\ln 2}{\chi}=\frac{\ln 2}{\ln 2+\bf g(0)}$.

\

It is now standard (see, e.g.~ \cite{pubook}, Theorem 8.6.5) to conclude that for $\mathbb P$ -a.e. $\omega\in\Omega$ $$\dim \mu_\omega=\frac{\ln 2}{\ln 2+{\bf g}(0)}.$$
\end{proof}

\

\section{Dimension of the maximal measure-dependence on a parameter}\label{sec:dependence}

In Section~\ref{sec:formula} we calculated Hausdorff dimension of the maximal measure on the Julia set of a randomly iterated polynomial.
More precisely, in Theorem~\ref{thm:dim_harm}, for  a bounded random system of independent quadratic maps, assuming  additionally that $\mathcal S$ is typically fast escaping,
we proved that  for $\mathbb P$-a.e. $\omega\in\Omega$ 

$$\dim_H(\mu_\omega)=\frac{\ln 2}{\chi}= \frac{\ln 2}{\ln 2+{\bf g}(0)}<1$$
Recall that here, $\mathbb P$ is the product distribution generated by an initial  distribution $\nu$  
in the parameter space. For example (see \cite{LZ}, Proposition 15), $\nu$ can be taken as the uniform distribution 
on a disc $\mathbb D(0,R)$ ($R>\frac{1}{4}$)  in the parameter space. Clearly, the global Green function $\bf g$ then depends on
the initial distribution $\nu$, in particular, in the above example the global Green function  $\bf g$ depends on $R$.

So, actually, the formula for the dimension of $\mu_\omega$ should be written as:
\begin{equation}\label{eq:dim_harm2}
\dim_H(\mu_\omega^R)=\frac{\ln 2}{\ln 2+{\bf g}_R(0)}
\end{equation}
to underline the dependence on $R$.

Recall also (see \cite{BBR}) that for every $\omega\in\overline{\mathbb D}(0,1/4)^{\mathbb N}$ $g_\omega(0)=0$ and  the set $K_\omega$ is connected, thus- the basin of infinity $\mathcal A_\omega$ is simply-connected. 
Therefore, for every 
$\omega\in\overline{\mathbb D}(0,1/4)^{\mathbb N}$ 
$\dim_H \mu_\omega=1$, 
as a consequence of Makarov's theorem (see, e.g., Theorems 5.1 and 5.2 in \cite{marshall})).
The function $R\mapsto {\bf g}_R(0)$ is also constant, equal to $0$ in the interval $[0,1/4]$.

\

This motivates the following question.

\begin{question}\label{quest:regul}
Examine regularity of  the function
$$[1/4,\infty)\ni R\mapsto {\bf g}_R(0).$$
\end{question}

Note that  regularity of the above function automatically implies the same kind of regularity of the function
$$\R\mapsto \dim_H(\mu_\omega^R)=\frac{\ln 2}{\ln 2+{\bf g}_R(0)}$$

\


\begin{remark}Let us compare this question to   the deterministic case, $z^2+c$.  Denoting by $\mu_c$ the harmonic (equilibrium) measure on the Julia set $J(f_c)=\partial \mathcal A_c$,  we have:
$$\dim_H(\mu)=\frac{\ln 2}{\ln 2+g_c(0)}$$
The function $c\mapsto g_c(0)$ is harmonic, thus-real analytic  in the complement of the Mandelbrot set.
 So,  also the function $c\mapsto\dim_H(\mu_c)$ is real- analytic in the complement of the Mandelbrot set.





\end{remark}

\subsection{Dependence of the  dimension  of the random maximal measure on the radius $R$}

In this section we give an initial answer to Question~\ref{quest:regul}.

From now on let us denote $\Omega_R = \mathbb D(0,R)^{\mathbb{N}}$, where $R>\frac{1}{4}$.

Choose $R_1<R_2$ such that $0<R_1<R<R_2$.  Pick $R_0$ large enough as in Proposition~\ref{prop:log} for $\Omega_{R_2}$. 
Then clearly this $R_0$ is an appropriate choice for Proposition~\ref{prop:log} for all $\Omega_R$ where $R \in (R_1, R_2)$.

\begin{lemma}\label{lem:bound}
There exists $M,h > 0$ such that the following holds true: if  $|z-z'| < h$ then for all  $R \in(R_1, R_2)$ and all $\omega, \omega ' \in \Omega_R$ we have
$$|g_{\omega} (z) - g_{\omega '} (z') | < M .$$
\end{lemma}

\begin{proof}
For $|z| \ge R_0+1$  we have $g_{\omega} (z) = \ln |z| + \varepsilon (\omega , z)$, where  $|\varepsilon (\omega, z)|\le 1$.
(see Proposition~\ref{prop:log}). Thus we have
$$
\aligned
|g_{\omega}(z) - g_{\omega'}(z')| &= |\ln(z) - \ln(z') + \varepsilon(\omega, z) - \varepsilon(\omega',z)| < 2 + |\ln(|z|) - \ln(|z'|)|\\
&=2+\left |\ln\left (1+\frac{|z'|-|z|}{|z|}\right )\right | < 3;
\endaligned
$$
the above holds for $h<1$ small enough.

Now, let $z \in \mathbb D_{R_0+1}$, and let $|z'-z|<1$. \\
Since  $z \in \mathbb D_{R_0+1}$ then $g_{\omega} (z)\le\ln|R_0+1|+1$, see again Proposition~\ref{prop:log}, and for $z'$ we have an analogous estimate with $R_0+1$ replaced by $R_0+2$.
Denoting $G:=\ln|R_0+2|+1$, we can estimate: 
$$|g_{\omega} (z) - g_{\omega '} (z')| < 2G,$$ and setting $M = \max (2G, 3)$ concludes the proof.
\end{proof}

For a sequence  $\omega=(c_i)_{i\in\mathbb N}\in \Omega_R$ we denote below  by $\omega(i)$ the parameter $c_i$. 
\begin{lemma}\label{lem:simseq}
For every $\varepsilon$, there exist $N\in\mathbb N$ and $\tilde h>0$ such that for all $R \in (R_1,R_2)$ and all 
$\omega, \omega' \in \Omega_R$ the following holds:

If $|\omega' (i) - \omega (i)| < \tilde h$ for all $i \leqslant N$, then $|g_{\omega} (0) - g_{\omega'} (0)| < \varepsilon$. 
\end{lemma}

\begin{proof}
Set $\varepsilon$ and let $N$ be such that
$$
\max \left (\frac{G}{2^N}, \frac{M}{2^N}\right ) < \varepsilon,
$$
where $M$ is chosen (along with a constant $h$) as in the previous lemma and $G$ is the constant which bounds $g_{\omega}(z)$ in $\mathbb D_{R_0}$. Finally let $\tilde{h}$
be small enough so that  

$$|\omega' (i) - \omega (i)| < \tilde h\quad\text{for all}\quad i \leqslant N \Rightarrow |f^N_{\omega}(0) - f^N_{\omega'}(0)| < h$$

where $h$ is chosen with $M$ as in Lemma \ref{lem:bound}.

We consider now  three cases: when $0$ is escaping for both $\omega$ and $\omega'$, when it is escaping for one of these sequences, and when it is escaping for neither.

Assume first that $0$ is escaping for both $\omega$ and $\omega'$. Let $z = f^N_{\omega}(0)$ and $z' = f^N_{\omega'} (0)$. 
Since $0$ is escaping, the formula $$g_{\omega}(0) = \lim\limits_{n \rightarrow \infty} \frac{1}{2^n} \ln |f^n_{\omega}(0)|$$ holds.
We have 
\begin{equation*}
    \begin{split}
        |g_{\omega} (0) - g_{\omega'}(0)| &= \lim\limits_{n \rightarrow \infty} \frac{1}{2^{n}}| \ln (|f^n_{\omega}(0)|) - \ln (|f^n_{\omega'} (0)|)|\\ &= \lim\limits_{n \rightarrow \infty} \frac{1}{2^{N}}|\frac{1}{2^n} \ln (|f^{n+N}_{\omega}(0)|) - \frac{1}{2^n} \ln (|f^{n+N}_{\omega'} (0)|)|\\ &=  \lim\limits_{n \rightarrow \infty} \frac{1}{2^{N}}(\frac{1}{2^n} \ln (|f^n_{\omega}(z)|) - \frac{1}{2^n} \ln (|f^n_{\omega'} (z')|))\\ &= \frac{1}{2^N} (g_{\omega} (z) - g_{\omega'}(z'))< \frac{M}{2^N}.
    \end{split}
\end{equation*} \\
Assume now that $0$ is escaping for $\omega$ but not for $\omega'$. This means $g_{\omega'} (0) = 0$, whereas $g_{\omega} (0) = \lim\limits_{n \rightarrow \infty} \frac{1}{2^{n}}| \ln (|f^n_{\omega}(0)|)>0$. Moreover $R_0 > |f^N_{\omega'} (0)|$ which implies $R_0 + h > |f^N_{\omega} (0)|$. Thus we have: 
\begin{equation*}
    \begin{split}
        |g_{\omega}(0) - g_{\omega'}(0)| &= \lim\limits_{n \rightarrow \infty} \frac{1}{2^{n}}|
        \ln (|f^n_{\omega}(0)|)\\ &=  \lim\limits_{n \rightarrow \infty} \frac{1}{2^{n+N}}| \ln (|f^{n+N}_{\omega}(0)|)\\ 
        &= \frac{1}{2^N} \lim\limits_{n \rightarrow \infty} \frac{1}{2^{n}}| \ln (|f^n_{\sigma^N\omega}(f^N_{\omega}(0))|)\\ 
        &= \frac{1}{2^N} g_{\sigma^N\omega} (f^N_{\omega}(0))< \frac{\tilde G}{2^N}.
    \end{split}
\end{equation*}\\
where $\tilde G=\ln|R_0+h|+1$ is the bound  of  $g_{\omega}(z)$ in $\mathbb D_{R_0+h}$.

Finally if $0$ is not escaping for $\omega, \omega'$, then the result holds trivially as $g_{\omega}(0) = g_{\omega'} (0) = 0$. This concludes the proof.
 \end{proof}


For a sequence $\omega=(c_1, c_2,\dots)$ let us denote by $(1+h)\omega$ the sequence $(1+h)\omega:=((1+h)c_1, (1+h)c_2 \dots)$. Then Lemma \ref{lem:simseq} yields the following as a corollary.

\begin{lemma}\label{lem:cont_green}
For every $R>0, \varepsilon>0$ there exists $h>0$ such that for every $\omega \in \Omega_R$ we have $|g_{\omega}(0) - g_{(1+h)\omega}(0)| < \varepsilon$.
\end{lemma}
\
We shall denote by $\mathbb P_R$ the product distribution of uniform 
distributions on $\mathbb D(0,R)$ in the product space $\mathbb D(0,R)^\mathbb N$.

\begin{theorem}\label{thm:continuous}
Let $\mathbb P_R$ the product distribution of uniform 
distributions on $\mathbb D(0,R)$ in the product space $\mathbb D(0,R)^
{\mathbb N}$. Let ${\bf g}_R$ be the global Green's fuction for the random 
system of quadratic maps generated by this distribution. 
Then  
the function $R\mapsto {\bf g}_R (0)$ is continuous in the interval  $[\frac{1}{4},\infty)$. Consequently, the function 

$$R\mapsto\dim_H\mu^R_\omega$$
is continuous in the interval $[\frac 1 4 ,+\infty)$.

\end{theorem}
\begin{proof}
Let $R\in [1/4,\infty)$,
choose some $R_1<R_2$ 
such that $1/4<R_1<R<R_2$.
We have
$$
{\bf g}_R(0) = \int_{\mathbb D(0,R)^{\mathbb{N}}} g_{\omega} (0) d \mathbb P_R (\omega) = \int_{\mathbb D (0,1)^{\mathbb{N}}} g_{R \omega} (0) d \mathbb P (\omega).
$$
where for $\omega=(c_0,c_1,\dots)\in\mathbb D(0,1)^{\mathbb N}$  we denoted by $R\omega$ the sequence $R\omega=(Rc_0,Rc_1,\dots)\in \mathbb D(0,R)^{\mathbb N}$.
Now from  Lemma~\ref{lem:cont_green} we conclude that for any $\varepsilon>0$ there exists $h_\varepsilon>0$ such that for all $h<h_\varepsilon$
we obtain
$$
|{\bf g}_R (0) - {\bf g}_{R +h}(0)| \leqslant \int_{\mathbb D (0,1)^{\mathbb{N}}}| g_{R \omega}(0) - g_{(R + h)\omega}(0) | d \mathbb P (\omega) < \varepsilon.
$$
\end{proof}

The next result gives the information about the asymptotics of the generalized Green function at infinity.
\begin{theorem}\label{thm:asymp}
Let $\mathbb P_R$ the product distribution of uniform 
distributions on $\mathbb D(0,R)$ in the product space $\mathbb D(0,R)^
{\mathbb N}$. Let ${\bf} g$ be the global Green's fuction for the random 
system of quadratic maps generated by this distribution. 

The generalized Green function ${\bf g}_R (0)$ satisfies
$\lim\limits_{R \rightarrow \infty} \frac{{\bf g}_{R}(0)}{\ln (R)} = \frac{1}{2}$.
Consequently,
$$\lim_{R\to\infty}{\dim_H(\mu^R_\omega)}\cdot {\ln R}=2 \ln(2).$$
\end{theorem}

\begin{proof}
First, note that for any $\omega \in \Omega_R$ we have $g_{\omega}(0) \leqslant g_{\omega_R} (0)$ where $\omega_R$ is the 
constant sequence defined by $\omega_R (i) = R$. Therefore, we have 
$$
{\bf g}_{R}(0) = \int_{\Omega_R} g_{\omega}(0) d\mathbb P_R (\omega)  \leqslant \int_{\Omega_R}  g_{\omega_R} (0) d \mathbb P_R(\omega) = g_{\omega_R} (0).
$$
We will now show a similar bound for ${\bf g}_R (0)$ from below.
Let us denote $$E_{R, \delta} = \{\omega \in \Omega_R : |\omega(0)| > \delta R \},$$
where $\delta \in (0,1)$. 
Then clearly the measure $\mathbb P_R(E_{R, \delta})$ is simply $1 - \delta^2$. Take any sequence $\omega' \in E_{R,\delta}$ for which $0$ is escaping (so almost every $\omega'$ will do). Let us denote $h_R (z) = z^2 - R$. Let us fix $\delta$ for now, and note that for $R$ large enough (chosen for this set $\delta$) we have
\begin{equation*}
    \begin{split}
        g_{\omega'}(0) & = \lim\limits_{n \rightarrow \infty} \frac{1}{2^n} \ln (|f^n_{\omega'}(0)|) \\
        & \geqslant \limsup\limits_{n \rightarrow \infty} \frac{1}{2^n} \ln (|h_R^{n-1} (\delta R)|) \\
        & \geqslant \lim\limits_{n \rightarrow \infty} \frac{1}{2^n} \ln ((\frac{\delta}{2} R)^{2^{n-1}}) \\
        & = \frac{1}{2} \ln (\delta R).
    \end{split}
\end{equation*}

The first equality is just a formula for Green's function which holds when $0$ is escaping, the first inequality follows from a simple 
induction argument. Indeed, by assumption we have
$$
|f_{\omega'}(0)| \geqslant \delta R
$$
and it is easy to see that we have
$$
|f^{n}_{\omega'}(0)| = |f^{n-1}_{\omega'}(0)^2 + c_n| \geqslant |f^{n-1}_{\omega'}(0)|^2 - R = |h_R (f^{n-1}_{\omega'}(0))|.
$$
Finally the last inequality is a simple induction argument. Indeed, we want to prove
$$
h_R^n(\delta R) > \left( \frac{\delta R}{2} \right)^{2^n}
$$
but assuming the above inequality for $n-1$ we get
$$
h_R^{n}(\delta R) = h_R (h_R^{n-1}(\delta R)) > \left(\frac{h_R^{n-1}(\delta R)}{2} \right)^2 > \left(\frac{\delta R}{2}\right)^{2^n}
$$
since for a large enough $z$ we have $z^2 - R > \left(\frac{z}{2}\right)^2$.
Thus we have
$$
g_{R}(0) = \int_{\Omega_R} g_{\omega}(0) d\mathbb P_R (\omega) \geqslant \int_{E_{R,\delta}} g_{\omega}(0) d\mathbb P_R (\omega) \geqslant \int_{E_{R,\delta}} \frac{1}{2} \ln (\delta R)
d\mathbb P_R (\omega) \geqslant \frac{1}{2}  (1 - \delta^2) \ln (\delta R).
$$

So we end up with
$$
\frac{1}{2}(1-\delta^2) \ln (\delta R) \leqslant g_R (0) \leqslant g_{\omega_R} (0).
$$
The left hand side yields 
$$
\liminf\limits_{R \rightarrow \infty} \frac{g_{R}(0)}{\ln (R)} \geqslant \lim\limits_{R \rightarrow \infty} \frac{(1- \delta^2)\ln(\delta R)}{2 \ln (R)} = \frac{1}{2}(1-\delta^2)
$$
and since $\delta$ can be chosen arbitrarily close to $0$ this concludes the proof in one direction. \\
Now for the other inequality. Note that there exists a constant $M$ such that for all $n$ the polynomial $h_R (z) = z^2 + R$ satisfies
$$
h_R^n(0) = h_R^{n-1} (R) < (M R)^{2^{n-1}}.
$$
Thus we have
$$
\limsup\limits_{R \rightarrow \infty} \frac{g_{R}(0)}{\ln (R)} \leqslant \limsup\limits_{R \rightarrow \infty} \frac{g_{\omega_R} (0)}{\ln (R)} \leqslant \lim\limits_{R \rightarrow \infty} \frac{\frac{1}{2}\ln (M R)}{\ln (R)} = \frac{1}{2}
$$
This concludes the proof.
\end{proof}

\begin{remark}
Let us note that for the autonomous iteration we have $g_c(0)=\frac{1}{2}\ln |c|+o(1)$ as $c\to\infty$, thus we have the same (actually: more precise) kind of asymptotics:
$$\lim_{c\to\infty} \dim_H(\mu_c)\cdot \ln|c|=2\ln 2.$$
\end{remark}

\subsection{Randomly perturbed parameter $c_0\notin\mathcal M$.}
In this subsection,
we consider a related problem. Namely, we study random iterates of the maps $z^2+c$ 
where $c$ is chosen independently at each step from the disc $\mathbb D(c_0,\delta_0)$ where $c_0\notin\mathcal M$ 
and $\delta_0$ 
is small enough, in particular  ${\rm dist}(c_0,\mathcal M)>\delta_0$. 

Again, the  dimension of the equilibrium  measure is constant almost everywhere. It is  given by the 
formula analogous to  \eqref{eq:dim_max} (see  Proposition~\ref{prop:dim_harm}).

Next, we study the dependence of this value on the ''range of randomness'' $\delta_0$.
This dependence turns out to be very regular.

\


More formally,   introduce the product  space:

$$\Omega=\mathbb D(0,1)^\mathbb N.$$
So, the space  $\Omega$ consists of infinite sequences of points in $\mathbb D(0,1)$, denoted by $\omega=(v_0,v_1,\dots)$.

Let $\nu$ be a probability distribution on $\mathbb D(0,1)$, and let 
$\mathbb P$ be the (completed) product  distribution on $\Omega$.
We now introduce the dependence on $\lambda$. 
For each $\lambda\in\Lambda:=\mathbb D(0,\delta_0)$  denote 
$$\Omega_\lambda=\lambda\cdot \Omega,$$
i.e, a point $\omega_\lambda\in\Omega_\lambda$,  is represented as $\lambda\cdot \omega=(\lambda v_0,\lambda v_1,\dots)$, where $\omega\in\Omega$.

For each $\omega=(v_0,v_1,\dots)$ and $\lambda\in\Lambda$ we write 
$c_n=c_n(\lambda,\omega)=c_0+\lambda v_n$, and write
 $f_{\lambda,\omega}$ to denote the map
\begin{equation}\label{eq:lambda1}
 f_{\lambda,\omega}= f_{c_0+\lambda v_0}.
\end{equation}
Analogously,
\begin{equation}\label{eq:lambda2}
f^n_{\lambda,\omega}=
f_{c_0+\lambda v_{n-1}}\circ\dots f_{c_0+\lambda v_2}\circ f_{c_0+\lambda v_1}\circ f_{c_0+\lambda v_0}=f_{c_{n-1}}\circ \dots \circ f_{c_0}.
\end{equation}

\

Put $R=|c_0|+{\rm dist}(c_0,\mathcal M)$. Then all parameters $c_0+\lambda v_j$ are in the disc $\mathbb D(0,R)$. 
Let $R_0$ be the value assigned to $R$ as in Proposition ~\ref{prop:log}.

\

We denote by  $J_{\lambda,\omega}$ the Julia set
for the non- autonomous sequence $f^n_{\lambda,\omega}$.

\

Let us note first the following stability properties (Proposition~\ref{prop:stability} and Lemma~\ref{lem:distance2}).
\begin{prop}\label{prop:stability}
 For every $\varepsilon>0$ there exists $\delta_0>0$ such that each Julia set
 $J_{\lambda,\omega}$ with $\lambda\in\mathbb D(0,\delta_0)$ is $\varepsilon$ -close to the Julia set $J(f_{c_0})$,
 i.e., 
 $$J_{\lambda,\omega}\subset B(J(f_{c_0}),\varepsilon).$$
\end{prop}
\begin{proof} 
Choose an arbitrary $\varepsilon>0$. 
 Note that on the set $\mathbb D(0,R_0)\setminus  B(J(f_{c_0}),\varepsilon)$ the iterates $f^n_{c_0}$ converge to $\infty$ uniformly.
 Therefore, one can choose a common value $n_0$ such that for every $z\in  \mathbb D(0,R_0)\setminus  B(J(f_{c_0}),\varepsilon)$ we have that 
 $|f^{n_0}_{c_0}(z)|>2R_0$.  Consequently, for $\delta_0$ small enough, every $z\in \mathbb D(0,R_0)\setminus  B(J(f_{c_0}),\varepsilon)$, 
 every $\lambda\in \mathbb D(0,\delta_0)$ and every $\omega\in\Omega$ the estimate
 $|f^{n_0}_{\lambda,\omega}(z)|>R_0$ holds. It now follows from the choice of $R_0$ that  $z\in \mathcal A_{\lambda,\omega}$.
 \end{proof}

\begin{lemma}\label{lem:distance2} Let $c_0\notin\mathcal M$, and let $f_{\lambda, \omega}$ be defined as in \eqref{eq:lambda1} and \eqref{eq:lambda2}.
For every $\delta_0$ small enough there exists $\eta>0$  such that for all $\omega\in\Omega$, $\lambda\in\Lambda$, 
 $J_{\lambda,\omega}\cap\mathbb D(0,\eta)=\emptyset$. In particular, the critical point $0$ is in the basin of infinity, $\mathcal A_\omega$, for every $\omega\in\Omega$.
\end{lemma}
\begin{proof}
Since $f^n_{c_0}(0)\to\infty$ as $n\to\infty$, 
there exists $n_0>0$ such that $|f^{n_0}_{c_0}(0)|>2 R_0$. 
Thus, for $\delta_0$ small enough, and every $\omega$,
$|f^{n_0}_{\lambda,\omega}(0)|>R_0$, and, by continuity, the same holds for every $z\in \mathbb D(0,\eta)$ with $\eta>0$ sufficiently small.
The conclusion now follows from the choice of $R_0$.
\end{proof}

Analogously to the notation introduced in  Section~\ref{sec:green}, we denote by $\mu_{\lambda,\omega}$ the equilibrium (harmonic)
measure on the Julia set $J_{\lambda,\omega}$ 
for the non- autonomous sequence $f^n_{\lambda,\omega}$.
As in Section~\ref{sec:formula} the following Proposition ~\ref{prop:dim_harm} holds true. 
Its  proof is essentially  easier than in the previous, non-hyperbolic case considered in Theorem ~\ref{thm:dim_harm}, because, as noticed in \cite{LZ} Remark in Section 6, the sequence  $k_n$ can be taken to be the  sequence of all (sufficiently large) integers.

\begin{prop}\label{prop:dim_harm}
Let $c_0\notin\mathcal M$. There exists $\delta_0>0$ (depending on the choice of $c_0$)  such that the following holds.
Choose an arbitrary $\lambda\in\Lambda=\mathbb D(0,\delta_0)$.
For $\mathbb P$-a.e. $\omega\in\Omega$ 

\begin{equation}\label{eq:dim_max2}
\dim_H(\mu_{\lambda,\omega})=\frac{\ln 2}{\chi}= \frac{\ln 2}{\ln 2+{{\bf g}_\lambda}(0)}<1
\end{equation}
where ${\bf g}_\lambda=\int_\Omega g_{\lambda,\omega}d\mathbb P(\omega)$.
\end{prop}


\begin{prop}\label{prop:green_harmonic2}
There exists $\delta_0>0$ such that for every $\omega\in \Omega$ the function
$$\Lambda\ni\lambda\mapsto g_{\lambda,\omega}(0)$$
is harmonic in $\Lambda=\mathbb D(0,\delta_o)$.
\end{prop}
\begin{proof}
Recall that for every $z\in\mathcal A_\omega$ the Green's function $g_{\lambda, \omega}(z)$ is given by the formula

$$g_{\lambda,\omega}(z)=\lim_{n\to\infty} \frac{1}{2^n}\ln|f^n_{
\lambda,\omega}(z)|.$$
Denote by $g_{n,\omega}$ the function

\begin{equation}\label{eq:green_sequence}
g_{n,\omega}(\lambda,z)= \frac{1}{2^n}\log|f^n_{
\lambda,\omega}(z)|.
\end{equation}

If $\lambda$ is in $\mathbb D(0,\delta_0)$,  as in Lemma~\ref{lem:distance2}, then $0\in\mathcal A_\omega$, so it holds for every $\lambda\in\Lambda$ that
\begin{equation}\label{eq:green_conv2}
g_{\lambda,\omega}(0)=\lim_{n\to\infty}g_{n,\omega}(\lambda,0).
\end{equation}
Now, observe that for each $n$ and for each $\omega$  the function

$$\Lambda\ni\lambda\mapsto g_{n,\omega}(\lambda, 0)$$ is a harmonic finction of the variable $\lambda$. Indeed, this follows from the formula ~\eqref{eq:green_sequence} and from the fact that the function
$\lambda\mapsto\log|f^n_{\lambda,\omega}(0)|)$ is harmonic for all $n>n_0$, since for all $n>n_0$, all $\lambda\in\Lambda$ and all $\omega\in\Omega$ we have that 
$$|f^{n}_{\lambda,\omega}(0)|>R_0.$$

So, to conclude that the function $\Lambda\ni\lambda\mapsto g_{\lambda,\omega}(0)$ is harmonic in $\Lambda=\mathbb D(0,\delta_0)$, it is enough to check that the convergence in \eqref{eq:green_conv2} is (locally) uniform.

To check this,
observe first  that for $z\in\mathcal A_\omega$ we have that
\begin{equation}\label{eq:green_uniform} 
 g_{n+1,\omega}(\lambda, z)=g_{n,\omega}(\lambda,z)+\frac {1}{2^{n+1}} 
 \ln \left |1+\frac {c_n}{(f^n_{\lambda,\omega}(z))^2}\right |
 \end{equation}
 where $c_n=c_n(\lambda, \omega)$.
 
 By Proposition~\ref{prop:log}, $|f_{\lambda,\omega}(w)|>2 |w|$ for every $w\in\mathbb D_{R_0}^*$ , which allows us to conclude that the convergence in \eqref{eq:green_uniform} is uniform in  $\Lambda=\mathbb D(0,\delta_0)$.

\end{proof}

We have the following corollary.
\begin{prop}\label{cor:green_harmonic3}
Let $\nu$ be a Borel probability measure on $\mathbb D(0,1)$. Let   $\mathbb P$ be the product distribution on $\Omega$ generated by $\nu$. Let $c_0\notin\mathcal M$ and let $\Lambda=\mathbb D(0,\delta_0)$ be as in Proposition~\ref{prop:green_harmonic2}.
 The function
$$\lambda\mapsto {\bf g}_\lambda(0)=\int_{\omega\in\Omega}g_{\lambda,\omega}(0)d\mathbb P(\omega)$$
is   harmonic  in $\Lambda$.
\end{prop}

Restricting the values of $\lambda$ to real ones, we obtain the following 
result on the dependence of the dimension of $\mu_{\delta_0,\omega}$ (which is constant for $\mathbb P$-a.e. $\omega$)  on the size of the ''randomness disc'' $\delta$.

\begin{prop}\label{prop:dim_analytic}
Let $\nu$ be a Borel probability measure on $\mathbb D(0,1)$. Let   $\mathbb P$ be the product distribution on $\Omega$ generated by $\nu$.
The function $\delta\mapsto {\bf g}_\delta(0)$  is real analytic  in the segment $[0,\delta_0)$.   More precisely: this function has a real analytic (even: harmonic)  extension to the disc $\mathbb D(0,\delta_0)$.
This implies, by Proposition~\ref{prop:max_measure2}, that the dimension 
$[0,\delta_0)\ni \delta\mapsto \dim\mu_{\delta,\omega}$ (the function which is constant for $\mathbb P$-a.e. $\omega\in\Omega$)
is real analytic in $[0,\delta_0)$.
\end{prop}

\medskip

Assuming that the initial distribution $\nu$ is uniform, we obtain a stronger result.

\begin{prop}\label{prop:stronger}
 Let $\nu$ be the uniform distribution  on $\mathbb D(0,1)$. Let   $\mathbb P$ be the product distribution on $\Omega$ generated by $\nu$.
For sufficiently small $\delta_0$ the function $\delta\mapsto {\bf g}_\delta(0)$  is constant in the segment $[0,\delta_0)$.
Consequently, for sufficiently small $\delta_0$ the function

$$[0,\delta_0)\ni \delta\mapsto \frac{\ln 2}{\ln 2+{\bf g}_\delta(0)}$$
is constant in $[0,\delta_0)$, which implies, by Proposition~\ref{prop:max_measure2}, that the dimension 

$[0,\delta_0)\ni \delta\mapsto \dim\mu_{\delta,\omega}$ (the function which is constant for $\mathbb P$-a.e. $\omega\in\Omega$)
is constant in $[0,\delta_0)$, and equal
to the dimension of $\mu_{c_0}$, i.e.
$$ \dim\mu_{\delta,\omega}=\mu_{c_0}=\frac{\ln 2}{\ln 2+g_{c_0}(0)}.$$
 \end{prop}
\begin{proof}
 We already know by Corollary ~\ref{cor:green_harmonic3} that the function 
 $$\lambda\mapsto {\bf g}_\lambda(0)=\int_{\omega\in\Omega}g_{\lambda,\omega}(0)d\mathbb P(\omega)$$
is harmonic on $\Lambda=\mathbb D(0,\delta_0)$.
 
Now, observe that, in the case of the product distribution $\mathbb P$ generated by the uniform distribution on  $\mathbb D(0,\delta_0)$, this function is also rotation invariant.
Indeed, let $\eta\in\mathbb S^1$ be an arbitrary rotation angle. Then

$${\bf g}_{\eta\cdot \lambda}(0)=\int_{\omega\in\Omega}g_{\eta\cdot \lambda,\omega}(0)d\mathbb P(\omega)= 
\int_{\omega\in\Omega}g_{\lambda,\eta\cdot\omega}(0)d\mathbb P(\omega), $$
where we denoted by $\eta\cdot \omega$ the sequence with all entries multiplied by $\eta$, i.e., 
if  $\omega=(v_0,v_1,\dots)$ the $\eta\cdot \omega=(\eta v_0, \eta v_1,\dots)$.
Now notice that the map
$E:\Omega\to\Omega$
defined by 
$$E(\omega)=\eta\cdot\omega$$
preserves the probability distribution $\mathbb P$ (here we use the fact that $\mathbb P$ is generated by the uniform distribution on the disc $\mathbb D(0,1)$).
Therefore,
$${\bf g}_{\eta\cdot \lambda}(0)=\int_{\omega\in\Omega}g_{\eta\cdot \lambda,\omega}(0)d\mathbb P(\omega)= \int_{\omega\in\Omega}g_{\lambda,E(\omega),}(0)d\mathbb P(\omega)
=\int_{\omega\in\Omega}g_{\lambda,\omega}(0)d\mathbb P(\omega)={\bf g}_{ \lambda}(0)$$

So, $\Lambda\ni \lambda\mapsto {\bf g}_\lambda(0)$ is harmonic in $\Lambda=\mathbb D(0,\delta_0)$ and rotation-invariant, thus- constant.
\end{proof}

\begin{remark}
Another random dimension to study is, of course, the Hausdorff dimension of the random Julia set itself.    
Denote by $\dim_H(J_{\delta,\omega})$ the value of the Hausdorff dimension of the random Julia set $J_{\delta,\omega}$.
The following theorem is a consequence of the result of  Rugh (see~\cite{rugh}).  

\begin{Theorem*}[see \cite{rugh}]
Let $\nu$ be a Borel probability measure on $\mathbb D(0,1)$. Let   $\mathbb P$ be the product distribution on $\Omega$ generated by $\nu$. 

There exists $\delta_0>0$ such that  for every $\delta\in [0,\delta_0)$ the function
$$\omega\mapsto \dim_H(J_{\delta, \omega})$$ is constant $\mathbb P$-- almost everywhere.  Denote this 
function by ${\bf d}_\delta$.
Then the function $\delta\mapsto {\bf d}_\delta$  is real analytic in the segment $[0,\delta_0)$. More precisely: this function has a real analytic extension to the disc $\mathbb D(0,\delta_0)$.
\end{Theorem*}

For every $c_0\notin\mathcal M$ it is known that there is a gap between the dimension of the harmonic measure and the dimension of the Julia set itself, i i.e.

$$\dim_H(\mu_{c_0})<\dim_H(J_{c_0}).$$

Now,   when the  the iteration becomes random, and the parameter $c$ is chosen in consecutive steps independently,  and  according to the uniform distribution on $\mathbb D(c_0, \delta)$, $\delta<\delta_0$ then the dimension of the harmonic measure remains (almost surely) unchanged,  while the dimension of the random Julia set varies (real) analytically with $\delta$.   This should motivate more detailed study of the  dependence of the dimension of the random Julia set on the parameter $\delta$.
\end{remark}

 \

\end{document}